\documentclass[12pt, a4paper]{article}
\usepackage{amsfonts,amssymb,amsmath,amscd,latexsym,makeidx,theorem}
\usepackage{hyperref}
\usepackage{color}

\newtheorem{thm}{Theorem}[section]
\newtheorem{lem}[thm]{Lemma}

\newtheorem{cor}[thm]{Corollary}
\newtheorem{defn}{Definition}[section]

\newenvironment{proof}{\noindent\emph{Proof.}}{\hfill$\square$\medskip}

\newcommand{\D}{\Delta}
\newcommand{\vp}{\varphi}
\newcommand{\R}{\mathbb{R}}
\newcommand{\ve}{\varepsilon}
\newcommand{\s}{\mathcal{S}}
\newcommand{\N}{\mathbb{N}}
 
\title{Conformally Euclidean  metrics on $\R^n$ with arbitrary  total $Q$-curvature}

\author{Ali Hyder \thanks{The author is supported by the Swiss National Science Foundation, project no. PP00P2-144669.} 
\\ {\small Universit\"at Basel}\\ {\small \texttt{ali.hyder@unibas.ch}}}

\begin{document}

\maketitle

\begin{abstract}
 We study the existence of solution to the problem 
 $$(-\D)^\frac n2u=Qe^{nu}\quad\text{in }\R^{n},\quad \kappa:=\int_{\R^{n}}Qe^{nu}dx<\infty,$$ where $Q\geq 0$, $\kappa\in (0,\infty)$ and $n\geq 3$. Using 
 ODE techniques Martinazzi for $n=6$ and Huang-Ye for $n=4m+2$ proved the existence of solution to the above problem with $Q\equiv const>0$ and for 
 every $\kappa\in (0,\infty)$. We extend these results in every dimension $n\geq 5$, thus completely answering the problem opened by Martinazzi. Our approach
 also extends to the case in which $Q$ is non-constant, and under some decay assumptions on $Q$ we can also treat the cases $n=3$ and $4$.
\end{abstract}

 \section{Introduction}
 
For a function $Q\in C^0(\R^n)$ we consider the problem
 \begin{align}\label{eq-1}
 (-\D)^\frac n2u=Qe^{nu}\quad\text{in }\R^{n},\quad \kappa:=\int_{\R^{n}}Qe^{nu}dx<\infty,
 \end{align}
  where for $n$ odd the non-local operator $(-\D)^\frac n2$ is defined in Definition \ref{def-operator}.
  
 Geometrically if $u$ is a smooth solution of \eqref{eq-1} then the conformal metric $g_u:=e^{2u}|dx|^2$ ($|dx|^2$ is the Euclidean metric on $\R^n$) has the $Q$-curvature $Q$. Moreover, the 
 total $Q$-curvature of the metric $g_u$ is $\kappa$.
 
 Solutions to \eqref{eq-1} have been classified in terms of there asymptotic behavior at infinity, more precisely we have the following: 
 
 \medskip
\noindent\textbf{Theorem A} (\cite{Chen-Li, DMR, Lin, LM, JMMX, Hyd, Xu}) \emph{Let $n\geq 1$.  Let $u$ be a solution of 
 \begin{align}\label{eq-2}
   (-\D)^\frac n2u=(n-1)!e^{nu}\quad\text{in }\R^{n},\quad \kappa:=(n-1)!\int_{\R^{n}}e^{nu}dx<\infty.
 \end{align}
 Then 
 \begin{align}\label{behavior}
  u(x)=-\frac{2\kappa}{\Lambda_1}\log |x|+P(x)+o(\log|x|),\text{ as }|x|\to\infty,
 \end{align}
 where $\Lambda_1:=(n-1)!|S^n|$,
 $o(\log|x|)/\log|x|\to0$ as $|x|\to\infty$ and $P$ is a polynomial of degree at most $n-1$ and $P$ is bounded from above.
 If $n\in\{3,4\}$ then $\kappa\in (0,\Lambda_1]$  and $\kappa=\Lambda_1$ if and only if $u$ is a spherical solution, that is, 
 \begin{align}
  u(x)=u_{\lambda,x_0}(x):=\log\frac{2\lambda}{1+\lambda^2|x-x_0|^2},  \label{spherical}
 \end{align}
 for some $x_0\in\R^n$ and $\lambda>0$. Moreover $u$ is spherical if and only if $P$ is constant (which is always the case when $n\in \{1,2\}$).}
 
 \medskip

 Chang-Chen \cite{CC} showed the existence of non-spherical solutions to \eqref{eq-2} in even dimension $n\geq 4$ for every $\kappa\in (0,\Lambda_1)$.
 
 A partial converse to Theorem A has been proven in dimension $4$ by  Wei-Ye \cite{W-Y} and extended by  Hyder-Martinazzi \cite{H-M} for $n\geq 4$ even and Hyder \cite{A-H} for $n\geq 3$. 
 
 \medskip
\noindent\textbf{Theorem B} (\cite{W-Y, H-M, A-H}) \emph{Let $n\geq 3$. Then for every $\kappa\in (0,\Lambda_1)$ and for every polynomial
$P$ with
$$
   deg(P)\leq n-1,\quad \text{and } P(x)\xrightarrow{|x|\to\infty}-\infty,
$$
 there exists a solution $u$ to \eqref{eq-2} having the asymptotic behavior given by \eqref{behavior}.}
 
 \medskip
 Although the assumption $\kappa\in (0,\Lambda_1]$ is a necessary condition for the existence of solution to \eqref{eq-2} for $n=3$ and $4$, it is possible
 to have a solution for $\kappa>\Lambda_1$ arbitrarily large in higher dimension as shown by Martinazzi \cite{LM-vol} for $n=6$. Huang-Ye \cite{HD} 
  extended  Martinazzi's result in arbitrary even dimension $n$ of the form $n=4m+2$ for some $m\geq 1$, proving that for every $\kappa\in(0,\infty)$ there exists a solution to 
 \eqref{eq-2}. The case $n=4m$ remained open. 
 
The ideas in \cite{LM-vol, HD} are based upon ODE theory. One considers only radial solutions so that the equation in \eqref{eq-2} becomes
an ODE, and the result is obtained by choosing suitable initial conditions and letting one of the
 parameters go to $+\infty$ (or $-\infty$). 
 However, this technique does not work if the dimension $n$ is a multiple of $4$, and things get even worse in odd dimension since $(-\D)^\frac n2$
 is nonlocal and ODE techniques cannot be used.

In this paper we extend the works of \cite{LM-vol, HD} and completely solve the cases left open, namely we prove that when $n\ge 5$
Problem \eqref{eq-2} has a solution for every $\kappa\in (0,\infty)$.  In fact we do not need to assume that $Q$ is constant, but only that
it is radially symmetric with growth at infinity suitably controlled, or not even radially symmetric. Moreover, we are able to 
 prescribe the asymptotic behavior of the solution $u$ (as in \eqref{behavior}) up to a polynomial of degree $4$ which cannot be prescribed
 and in particular it cannot be required to vanish when $\kappa \ge \Lambda_1$. This in turn, together with Theorem A, is consistent with
 the requirement $n\ge 5$, because only when $n\ge 5$ the asymptotic expansion of $u$ at infinity admits polynomials of degree $4$.

 We prove the following two theorems.
\begin{thm}\label{thm-1}
Let $n\geq 5$ be an  integer. Let $P$ be a polynomial on $\R^{n}$ with degree at most $n-1$. Let $Q\in C^0(\R^{n})$ be such that $Q(0)>0$, $Q\geq0$, $Qe^{nP}$ is 
radially symmetric and
$$\sup_{x\in \R^{n}}Q(x)e^{nP(x)}<\infty.$$  Then for every $\kappa>0$ there exists a 
  solution $u$ to \eqref{eq-1} such that 
  $$u(x)=-\frac{2\kappa}{\Lambda_1}\log|x|+P(x)+c_1|x|^2-c_2|x|^4+o(1),\quad\text{as }|x|\to\infty,$$ for some $c_1,c_2>0$. In fact, there exists 
  a radially symmetric function $v$ on $\R^{n}$ and a constant $c_v$ such that
   $$ v(x)=-\frac{2\kappa}{\Lambda_1}\log|x|+\frac{1}{2n}\D v(0)(|x|^4-|x|^2)+o(1),\quad\text{as }|x|\to\infty,$$ and 
   $$u=P+v+c_v-|x|^4,\quad x\in \R^{n}.$$ 
 \end{thm}
 
 Taking $Q=(n-1)!$ and $P=0$ in Theorem   \ref{thm-1} one has the following corollary. 
 \begin{cor}
 Let  $n\geq 5$. Let $\kappa\in (0,\infty)$. Then there exists a radially symmetric  solution $u$ to \eqref{eq-2} such that 
 $$u(x)=-\frac{2\kappa}{\Lambda_1}\log|x|+c_1|x|^2-c_2|x|^4+o(1),\quad\text{as }|x|\to\infty,$$ for some $c_1,c_2>0$.
 \end{cor}

 Notice that the polynomial part of  the solution $u$ in Theorem \ref{thm-1} is not exactly the prescribed polynomial $P$ 
 (compare \cite{W-Y, H-M, A-H}).  In general, without perturbing the polynomial part,  it is not possible to find a solution for
 $\kappa\ge \Lambda_1$. For example, if $P$ is non-increasing and non-constant then 
 there is no solution $u$ to \eqref{eq-2} with $\kappa\geq \Lambda_1$ such that $u$ has the asymptotic behavior \eqref{behavior} 
 (see Lemma \ref{non2} below). This justifies the term $c_1|x|^2$ in Theorem \ref{thm-1}. Then the additional term $-c_2|x|^4$
 is also necessary to avoid that $u(x)\ge \frac{c_1}{2}|x|^2$ for $x$ large, which would contrast with the condition $\kappa<\infty$, 
 at least if $Q$ does not decays fast enough at infinity.
 In the latter case, the term $-c_2|x|^4$ can be avoided, and one obtains an existence result also in dimension $3$ and $4$.
  
\begin{thm}\label{thm-2}
Let $n\geq 3$.
 Let $Q\in C^0_{rad}(\R^n)$ be such that $Q\geq0$, $Q(0)>0$ and 
 $$\int_{\R^n}Q(x)e^{\lambda |x|^2}dx<\infty,\quad\text{for every }\lambda>0, \quad\int_{B_1(x)}\frac{Q(y)}{|x-y|^{n-1}}dy\xrightarrow{|x|\to\infty}0.$$  
 Then for every $\kappa>0$ there exists a radially symmetric solution  $u$ to \eqref{eq-1}. 
\end{thm}

The decay assumption on $Q$ in Theorem \ref{thm-2} is sharp in the sense that if $Qe^{\lambda |x|^2}\not\in L^1(\R^n)$ for 
some $\lambda>0$, then Problem \eqref{eq-1}  might not have a solution for every   $\kappa>0$ . For instance,
if $Q= e^{-\lambda|x|^2}$ for some $ \lambda>0$,
then   \eqref{eq-1} with $n=3,4$ and  $\kappa>\Lambda_1$ has no radially symmetric solution (see Lemma \ref{non} below).

\medskip

The proof of Theorem \ref{thm-1} is based on the Schauder fixed point theorem, and the main difficulty is to show that the
``approximate solutions''  are pre-compact (see in particular Lemma \ref{2nd-max}). We will do that using blow up analysis
(see for instance \cite{ARS, dru, Mar1, Rob}). In general, if $\kappa\geq \Lambda_1$ one can expect blow up, but we will 
construct our approximate solutions carefully in a way that this does not happen. For instance in \cite{W-Y, H-M} one looks for solutions
of the form $u= P+ v+c_v$ where $v$ satisfies the integral equation
$$v(x)=\frac{1}{\gamma_n}\int_{\R^n}\log\left(\frac{1}{|x-y|}\right) Q(y)e^{nP(y)}e^{n(v(y+c_v)}dy,$$
and $c_v$ is a constant such that
$$\int_{\R^n}Q e^{n(P+v+c_v)}dx=\kappa.$$
With such a choice we would not be able to rule out blow-up. Instead, by looking for solutions of the form
$$u=P+v+P_v+c_v$$
where a posteriori $P_v=-|x|^4$, $v$ satisfies 
\begin{equation}\label{eqv}
v(x)=\frac{1}{\gamma_n}\int_{\R^{n}}\log\left(\frac{1}{|x-y|}\right)Q(y)e^{n(P(y)+P_v(y)+v(y)+c_v)}dy +\frac{1}{2n}(|x|^2-|x|^4)|\D v(0)|,
\end{equation}
and $c_v$ is again a normalization constant, one can prove that the integral equation \eqref{eqv} enjoys sufficient compactness,
essentially due to the term $\frac{1}{2n}|x|^2|\Delta v(0)|$ on the right-hand side. Indeed a sequence of (approximate) solutions
$v_k$ blowing up (for simplicity) at the origin, up to rescaling, leads to a sequence $(\eta_k)$ of functions satisfying for every $R>0$
$$\int_{B_R}|\Delta \eta_k -c_k|dx \le CR^{n-2}+o(1)R^{n+2},\quad o(1)\xrightarrow{k\to\infty}0,\quad c_k>0,$$
and converging to $\eta_\infty$ solving (for simplicity here we ignore some cases)
$$(-\D)^\frac n2\eta_\infty=e^{n\eta_\infty}\quad\text{in }\R^n,\quad \int_{\R^n}e^{n\eta_\infty}dx<\infty,$$
and 
\begin{equation}
 \int_{B_R}|\Delta \eta_\infty -c_\infty|dx\leq CR^{n-2},\quad c_\infty\geq 0, \label{est}
\end{equation}
 where $c_\infty=0$ corresponds to
$\D\eta_\infty(0)=0$ (see Sub-case 1.1 in Lemma \ref{2nd-max} with $x_k=0$). 
The  estimate on $\|\D \eta_\infty\|_{L^1(B_R)}$ in \eqref{est} shows that the polynomial part $P_\infty$ of $\eta_\infty$ (as in \eqref{behavior}) 
has degree at most $2$, and hence $\D P_\infty\leq 0$ as $P_\infty$ is bounded from above. Therefore,  $c_\infty=0=\D P_\infty$, and in particular $\eta_\infty$
is a spherical solution, that is, $\eta_\infty=u_{\lambda, x_0}$ for some $\lambda>0$ and $x_0\in\R^n$, where $u_{\lambda, x_0}$ is given by 
\eqref{spherical}. This leads to a contradiction as $\D\eta_\infty(0)=0$ and $\D u_{\lambda,x_0}<0$ in $\R^n$.   

\medskip 

In this work we focus only on the case $Q\ge 0$ because the negative case has been relatively well understood. For instance by a simple
application of maximum principle one can show that
 Problem \eqref{eq-1} has no solution with $Q\equiv const<0$, $n=2$ and $\kappa>-\infty$, 
but when $Q$ is non-constant, solutions do exist, as shown by Chanillo-Kiessling in \cite{Chanillo} under suitable assumptions. Martinazzi \cite{LM-Neg} 
 proved that in higher even dimension $n=2m\geq 4$ Problem \eqref{eq-1} with $Q\equiv const<0$ has solutions for some $\kappa$, and it has been shown in \cite{H-M} 
 that actually for every $\kappa\in (-\infty,0)$ and $Q$ negative constant \eqref{eq-1} has a solution. The same result has been recently extended to odd dimension $n\ge 3$ in \cite{A-H}.


\section{Proof of Theorem \ref{thm-1}}
 
 We consider the space
 $$X:=\left\{v\in C^{n-1}(\R^{n}): v\text{ is radially symmetric}, \|v\|_X<\infty\right\},$$ where 
 $$\|v\|_X:=\sup_{x\in\R^n}\left(\sum_{|\alpha|\leq 3}(1+|x|)^{|\alpha|-4}|D^\alpha v(x)|+\sum_{3<|\alpha|\leq n-1}|D^\alpha v(x)|\right).$$ 
 For $v\in X$ we set
 $$A_v:=\max\left\{0,\sup_{|x|\ge 10}\frac{v(x)-v(0)}{|x|^4}\right\},\quad P_v(x):=-|x|^4-A_v|x|^4.$$ 
  Then $$v(x)+P_v(x)\leq v(0)-|x|^4,\quad \text{for }|x|\geq 10.$$
 Let $c_v$ be the constant determined by
  $$\int_{\R^{n}}Ke^{n(v+c_v)}dx=\kappa,\quad K:=Qe^{nP}e^{nP_v},$$ where  the functions $Q$ and $P$ satisfy the hypothesis in Theorem \ref{thm-1}.
 Since $Q(0)>0$, without loss of generality we can also assume that $Q>0$ in $B_3$. 
    Then  $u=P+P_v+v+c_v$ satisfies $$(-\D)^\frac n2u=Qe^{nu}, \quad \kappa=\int_{\R^{n}}Qe^{nu}dx, $$ if and only if 
 $v$ satisfies $$(-\D)^\frac n2v=Ke^{n(v+c_v)}.$$
 For odd integer $n$, the operator $(-\D)^\frac n2$ is defined as follows:
  \begin{defn} \label{def-operator}
 Let $n$ be an odd integer.   Let $f\in \s'(\R^n)$. We say that $u$ is a solution of 
 $$(-\D)^\frac n2u=f\quad\text{in }\\R^n,$$ if $u\in W^{n-1,1}_{loc}(\R^n)$ and  $\D^\frac{n-1}{2}u\in L_\frac 12(\R^n)$ and for every
 test function $\vp\in\s(\R^n)$ 
 $$\int_{\R^n}(-\D)^\frac{n-1}{2}u(-\D)^\frac 12\vp dx=\langle f,\vp \rangle.$$  
 Here, $\s(\R^n)$ is the Schwartz space and the space $L_s(\R^n)$ is defined  by 
 $$L_s(\R^n):=\left\{u\in L^1_{loc}(\R^n): \|u\|_{L_s(\R^n)}:= \int_{\R^n}\frac{|u(x)|}{1+|x|^{n+2s}}dx<\infty\right\},\quad s>0.$$
 \end{defn}
  
 For more details on fractional Laplacian we refer the reader to \cite{Valdinoci}.
 
 \medskip
 We define an operator $T:X\to X$ given by $T(v)=\bar{v},$ where
 $$\bar{v}(x)=\frac{1}{\gamma_n}\int_{\R^{n}}\log\left(\frac{1}{|x-y|}\right)K(y)e^{n(v(y)+c_v)}dy+
            \frac{1}{2n}(|x|^2-|x|^4)|\D v(0)|,$$
            where $\gamma_n:=\frac{(n-1)!}{2}|S^n|$.

\begin{lem}\label{lim-sup}
 Let $v$ solve $tT(v)=v$ for some $0<t\leq1$. Then 
\begin{align}\label{def-v}
v(x)=\frac{t}{\gamma_n}\int_{\R^{n}}\log\left(\frac{1}{|x-y|}\right)K(y)e^{n(v(y)+c_v)}dy +\frac{t}{2n}(|x|^2-|x|^4)|\D v(0)|, 
\end{align}
 $\D v(0)<0$, and $v(x)\to-\infty$ as $|x|\to\infty$. Moreover,  $$\sup_{x\in B_1^c}v(x)\leq \inf_{x\in B_1}v(x),$$ and  in particular $A_v=0$.
 \end{lem}
\begin{proof}
 Since $v$ satisfies $tT(v)=v$, \eqref{def-v}  follows from the definition of $T$. Differentiating under integral sign, from \eqref{def-v} one can get  $\D v(0)<t|\D v(0)|$, which implies that 
 $\D v(0)<0$. The remaining part of the lemma follows from the fact that 
 \begin{equation}\label{Deltav}
\D v(x)<\frac{t}{2n}|\D v(0)|\D(|x|^2-|x|^4),\quad x\in\R^n,
\end{equation}
and  the integral representation of  radially symmetric functions given by
   \begin{align}\label{integlap}
  v(\xi)-v(\bar{\xi})&=\int_{\bar{\xi}}^\xi\frac{1}{\omega_{n-1}r^{n-1}}\int_{B_r}\D v(x)dxdr,\quad 0\leq\bar{\xi}<\xi,\quad \omega_{n-1}:=|S^{n-1}|.
 \end{align}
 \end{proof}

\begin{lem}\label{2nd-max}
Let $(v,t)\in X\times (0,1]$ satisfy $v=tT(v)$.  Then there exists $C>0$ (independent of $v$ and $t$)
such that $$\sup_{B_\frac18} w\leq C,\quad w:=v+c_{v}+\frac{1}{n}\log t.$$
\end{lem}
\begin{proof}
 Let us assume by contradiction that the conclusion of the lemma is false. Then there exists a sequence 
 $w_k=v_k+c_{v_k}+\frac1n\log t_k$ such that $\max_{\bar{B}_\frac18}w_k=:w_k(\theta_k)\to \infty $. 
 
 If $\theta_k$ is a point of local maxima of $w_k$ then  we set $x_k=\theta_k$. Otherwise, we can choose $x_k\in B_{\frac14}$
 such that $x_k$ is a point of local maxima of $w_k$ and $w_k(x_k)\geq w_k(x)$ for every $x\in B_{|x_k|}$. This follows from the fact that 
  $$\inf_{B_{\frac{1}{4}}\setminus B_\frac18}w_k\not\to\infty,$$ which is a consequence of 
 $$\int_{\R^n}Ke^{nw_k}dx=t_k\kappa\leq\kappa,\quad K>0\,\text{ on }B_3.$$
 
 We set $\mu_k:=e^{-w_k(x_k)}$. We distinguish the following cases. 
 
  \medskip
\noindent\textbf{Case 1}  Up to a subsequence $t_k\mu_k^2|\D v_k(0)|\to c_0\in [0,\infty).$ 

We set
 $$\eta_k(x):=v_k(x_k+\mu_kx)-v_k(x_k)=w_k(x_k+\mu_kx)-w_k(x_k).$$ 
 Notice that by \eqref{def-v} we have for some dimensional constant $C_1$
\begin{align}
\Delta\eta_k(x)&=\mu_k^2\Delta v_k(x_k+\mu_k x)\notag\\ 
&=C_1\frac{\mu_k^2}{\gamma_n}\int_{\R^n}\frac{K(y)e^{nw_k(y)}}{|x_k+\mu_k x -y|^2}dy +t_k\mu_k^2\left(1 -\frac{4(n+2)}{2n}|x_k+
\mu_k x|^2\right) |\Delta v_k(0)|, \notag 
\end{align}
 so that 
  \begin{align}
  &\int_{B_R}\left|\D\eta_k(x)-t_k\mu_k^2|\D v_k(0)|\left(1-\frac{2(n+2)}{n}|x_k|^2\right)\right|dx \notag\\
  &\leq \frac{C_1}{\gamma_n} \int_{\R^{n}}K(y)e^{nw_k(y)}\int_{B_R}\frac{\mu_k^2dx}{|x_k+\mu_kx-y|^2}dy
            +Ct_k\mu_k^2|\D v_k(0)|\int_{B_R}(\mu_k|x_k\cdotp x|+\mu_k^2|x|^2)dx\notag\\
  &\le  \frac{C_1}{\gamma_n}t_k \kappa \int_{B_R} \frac{1}{|x|^2}dx + Ct_k\mu_k^2|\D v_k(0)|\int_{B_R}(\mu_k|x|+\mu_k^2|x|^2)dx\notag \\
    &\leq C \kappa t_kR^{n-2}+Ct_k\mu_k^2|\D v_k(0)|\left(\mu_kR^{n+1}+\mu_k^2R^{n+2}\right). \label{lap-int}
 \end{align}
The function  $\eta_k$ satisfies 
$$(-\D)^\frac n2\eta_k(x)=K(x_k+\mu_kx)e^{n\eta_k(x)}\quad \text{in }\R^{n},\quad \eta_k(0)=0.$$
Moreover,  $\eta_k\leq C(R)$ on $B_R$. This follows easily if $|x_k|\leq \frac 19$ as in that case 
$\eta_k\leq 0$ on $B_R$ for $k\geq k_0(R)$. On the other hand, for $\frac19<|x_k|\leq \frac14$ one  can use 
 Lemma \ref{local-bound} (below). 
Therefore, by Lemma \ref{sch} (and Lemmas \ref{odd}, \ref{L1/2conv} if $n$ is odd),  up to a subsequence, $\eta_k\to\eta$ in $C^{n-1}_{loc}(\R^n)$ where $\eta$ satisfies 
$$(-\D)^\frac n2\eta=K(x_\infty)e^{n\eta}\quad \text{in }\R^{n},\quad K(x_\infty)\int_{\R^n}e^{n\eta}dx\leq t_\infty\kappa<\infty,\,K(x_\infty)>0,$$ 
where (up to a subsequence) $t_k\to t_\infty$ and $x_k\to x_\infty$.
Notice that $t_\infty\in (0,1]$, $x_\infty\in \bar B_\frac14$
and for every $R>0$, by \eqref{lap-int}
\begin{align}
 \int_{B_R}\left|\D\eta-c_0c_1\right|dx\leq CR^{n-2},\quad c_1=:1-\frac{2(n+2)}{n}|x_\infty|^2>0.\label{eta}
\end{align}
Hence by Theorem A we have 
$$\eta(x)=P_0(x)-\alpha\log|x|+o(\log|x|),\quad\text{as }|x|\to\infty,$$ where $P_0$ is a polynomial of degree at most $n-1$, $P_0$ is bounded from above and 
$\alpha$ is a positive constant. In fact, by \eqref{eta}
\begin{align*}
\int_{B_R}|\D P_0(x)-c_0c_1|dx\leq CR^{n-2}, \quad \text{for every }R>0.
\end{align*}
Hence $P_0$ is a constant. This implies that  $\eta$ is a spherical solution and in particular $\D\eta<0$ on $\R^n$, and therefore, again by 
\eqref{eta}, we have $c_0=0$.

We consider the following sub-cases. 

\noindent\textbf{Sub-case 1.1} There exists $M>0$ such that $\frac{|x_k|}{\mu_k}\leq M$. 

We set $y_k:=-\frac{x_k}{\mu_k}$. Then (up to a subsequence) $y_k\to y_\infty\in B_{M+1}$. Therefore, 
 $$\D\eta(y_\infty)=\lim_{k\to\infty}\D\eta_k(y_k)=\lim_{k\to\infty}\mu_k^2\D v_k(0)=\frac{c_0}{t_\infty}=0,$$ 
 a contradiction as $\D\eta<0$ on $\R^n$.
 
 \noindent\textbf{Sub-case 1.2} Up to a subsequence $\frac{|x_k|}{\mu_k}\to\infty$.
 
 For any $N\in\N$ we can choose $\xi_{1,k},\dots,\xi_{N,k}\in\R^n$ such that $|\xi_{i,k}|=|x_k|$ for all $i=1,\dots, N$ and 
 the balls $B_{2\mu_k}(\xi_{i,k})$'s are disjoint for $k$ large enough. Since $v_k$'s are radially symmetric, the functions 
 $\eta_{i,k}:=v_k(\xi_{i,k}+\mu_kx)-v_k(\xi_{i,k})\to \eta_i=\eta$ in $C^{n-1}_{loc}(\R^n)$. Therefore, 
 \begin{align}
 \lim_{k\to\infty} \int_{B_1}e^{n(v_k+c_{v_k})}dx\geq N\lim_{k\to\infty}\int_{B_{\mu_k}(\xi_{1,k})}e^{n(v_k+c_{v_k})}dx=N\frac{1}{t_\infty}\int_{B_1}e^{n\eta}dx.\notag
 \end{align}
This contradicts to the fact that $$\int_{B_1}Ke^{n(v_k+c_{v_k})}dx\leq\kappa,\quad K>0\,\text{ on }B_3.$$

\medskip
 \noindent\textbf{Case 2} Up to a subsequence $t_k\mu_k^2|\D v_k(0)|\to\infty.$ 
 
 We choose $\rho_k>0$ such that $t_k\rho_k^2\mu_k^2|\D v_k(0)|=1$. We set $$\psi_k(x)=v_k(x_k+\rho_k\mu_kx)-v_k(x_k).$$
 Then one can get (similar to \eqref{lap-int}) 
 \begin{align*}
  &\int_{B_R}\left|\D\psi_k(x)-\left(1-\frac{2(n+2)}{n}|x_k|^2\right)\right|dx \notag\\
  &\leq C_1 \int_{\R^{n}}K(y)e^{nw_k(y)}\int_{B_R}\frac{\rho_k^2\mu_k^2}{|x_k+\mu_k\rho_kx-y|^2}dxdy
            +C_2\mu_k\rho_k\int_{B_R}(|x|+\mu_k\rho_k|x|^2)dx\notag\\
  &\xrightarrow{k\to\infty}0,          
 \end{align*}
thanks to Lemma  \ref{int1} (below). 
 Moreover, together with Lemma \ref{local-bound}, $\psi_k$ satisfies 
$$(-\D)^\frac n2\psi_k=o(1)\quad\text{in }B_R,\,\psi_k(0)=0,\,\psi_k\leq C(R)\,\text{ on }B_R.$$
Hence, by Lemma \ref{sch} (and Lemma \ref{odd} if $n$ is odd), up to a subsequence  $\psi_k\to\psi$ in $C^{n-1}_{loc}(\R^n)$. Then $\psi$ must satisfy 
$$\int_{B_1}|\D\psi-c_0|dx=0,\quad c_0:=1-\frac{2(n+2)}{n}|x_\infty|^2>0,$$ where (up to a subsequence) $x_k\to x_\infty$.
This shows that $\D\psi(0)=c_0>0$, which is a contradiction as 
$$\D\psi(0)=\lim_{k\to\infty}\D\psi_k(0)=\lim_{k\to\infty}\rho_k^2\mu_k^2\D v_k(x_k)\leq 0.$$ 
Here, $\D v_k(x_k)\leq 0$ follows from the fact that $x_k$ is a point of local maxima of $v_k$.
 \end{proof}

A consequence of the local uniform upper bounds of $w$ is the following global uniform upper bounds: 
 \begin{lem}\label{uniform-v}
  There exists a constant $C>0$ such that for all $(v,t)\in X\times(0,1]$ with $v=tT(v)$ we have  $|\D v(0)|\leq C$ and 
  $$v(x)+c_v+\frac{1}{n}\log t\leq C,\quad\text{on }\R^{n}.$$ 
 \end{lem}
\begin{proof}
By Lemma \ref{2nd-max} we have $$\sup_{B_\frac18}w:=\sup_{B_\frac18}\left(v+c_v+\frac{1}{n}\log t\right)\leq C.$$
  Differentiating under integral sign from  \eqref{def-v} we obtain
 \begin{align*}
  |\D v(0)|&\leq C\int_{B_\frac18}\frac{1}{|y|^2}K(y) e^{nw(y)}dy+C\int_{B^c_\frac18}\frac{1}{|y|^2}K(y) e^{nw(y)}dy\\
  &\leq C\sup_{B_\frac18} K \int_{B_\ve}\frac{1}{|y|^2}dy+C\int_{B_\frac18^c}Ke^{nw}dy\\
  &\leq C(\ve,\kappa,K).
 \end{align*}
By \eqref{Deltav} we get
$$\D v(x)\leq t|\D v(0)|\leq C,\quad x\in \R^n,$$
 and hence, together with \eqref{integlap} 
 \begin{align}
  w(x)=w(0)+\int_{0}^{|x|}\frac{1}{\omega_{n-1}r^{n-1}}\int_{B_r}\D v(y)dydr\leq w(0)+C|x|^2\leq C,\quad x\in B_2.\notag
 \end{align}
The lemma follows from Lemma \ref{lim-sup}. 
\end{proof}

\medskip
\noindent\emph{Proof of Theorem \ref{thm-1}}
Let $v\in X$ be a solution of $v=tT(v)$ for some $0<t\leq 1$. 
Then $A_v=0$ and $|\D v(0)|\leq C$, thanks to  Lemmas \ref{lim-sup} and \ref{uniform-v}. Hence, for $0\leq|\beta|\leq n-1$
\begin{align*}
|D^\beta v(x)|&\leq C\int_{\R^n}\left|D^\beta\log\left(\frac{1}{|x-y|}\right)\right|K(y)e^{n(v(y)+c_v+\frac1n\log t)}dy+C|D^\beta (|x|^2-|x|^4)|\\
&\leq C\int_{\R^n}\left|D^\beta\log\left(\frac{1}{|x-y|}\right)\right|e^{-|y|^4}dy+C|D^\beta (|x|^2-|x|^4)|,
\end{align*}
where in the second inequality we have used  that $$v(x)+c_v+\frac{1}{n}\log t\leq C,\quad\text{ $C$ is independent of $v$ and $t$},$$ which follows from Lemma \ref{uniform-v}.
Now as in Lemma \ref{T-compact} one can show that $$\|v\|_X\leq M,$$ and therefore, by Lemma \ref{leray}, the operator $T$ has a fixed point (say) $v$.
Then $$u=P+v+c_v-|x|^4,$$ is  a solution to the Problem \eqref{eq-1} and $u$ has the asymptotic behavior given by
$$u(x)=P(x)-\frac{2\kappa}{\Lambda_1}\log|x|+\frac{1}{2n}\D v(0)(|x|^4-|x|^2)-|x|^4+c_v+o(1),\quad\text{as }|x|\to\infty.$$
This completes the proof of Theorem \ref{thm-1}.
\hfill $\square$

\medskip

Now we give a proof of the technical lemmas used in the proof of Lemma \ref{2nd-max}.
\begin{lem}\label{local-bound} Let $\ve>0$. Let $(v_k,t_k)\in X\times (0,1]$ satisfy \eqref{def-v} or \eqref{def-v4} for all $k\in\N$.
Let  $x_k\in B_1\setminus B_\ve$ be a point of maxima of $v_k$ on $\bar{B}_{|x_k|}$  and $v_k'(x_k)=0$. Then 
 \begin{align}
  v_k(x_k+x)-v_k(x_k)\leq C(n,\ve)|x|^2t_k|\D v_k(0)|,\quad x\in B_1.\notag
 \end{align}
\end{lem}
\begin{proof}
  If $|x_k+x|\leq |x_k|$ then $ v_k(x_k+x)-v_k(x_k)\leq 0$ as $v_k(x_k)\geq v_k(y)$ for every $y\in B_{|x_k|}$. 
 For $|x_k|<|x_k+x|$, setting $a=a(k,x):=x_k+x$, and together with \eqref{integlap}  we obtain
\begin{align*}
 v_k(x_k+x)-v_k(x_k)&=\int_{|x_k|}^{|a|}\frac{1}{\omega_{n-1}r^{n-1}}\int_{B_r\setminus B_{|x_k|}}\D v_k(x)dxdr\\
 &\leq \int_{|x_k|}^{|a|}\frac{1}{\omega_{n-1}r^{n-1}}\int_{B_{|a|}\setminus B_{|x_k|}}t_k|\D v_k(0)|dxd\rho\\
  &\leq C(n)t_k|\D v_k(0)|(|B_{|a|}|-|B_{|x_k|}|)\left(\frac{1}{|x_k|^{n-2}}-\frac{1}{|a|^{n-2}}\right)\\
  &\leq C(n,\ve)t_k|x|^2|\D v_k(0)|,
  \end{align*}
  where in the first equality we have used that $$0=v_k'(x_k)=\frac{1}{\omega_{n-1}|x_k|^{n-1}}\int_{B_{|x_k|}}\D v_kdx.$$ 
Hence we have the lemma.
\end{proof}

\begin{lem}\label{int1}
  Let $(v_k,t_k)\in X\times (0,1]$ satisfy \eqref{def-v} for all $k\in\N$.
 Let  $x_k\in B_1$ be a point of maxima of $v_k$ on $\bar{B}_{|x_k|}$  and $v_k'(x_k)=0$.
   We set $w_k=v_k+c_{v_k}+\frac{1}{n}\log t_k$
 and  $\mu_k=e^{-w_k(x_k)}$. 
  Let $\rho_k>0$ be such that   $t_k\rho_k^2\mu_k^2|\D v_k(0)|\leq C$ and $\rho_k\mu_k\to 0$. Then for any $R_0>0$
 $$\lim_{k\to\infty}\int_{\R^{n}}K(y)e^{nw_k(y)}\int_{B_{R_0}}\frac{\rho_k^2\mu_k^2}{|x_k+\rho_k\mu_kx-y|^2}dxdy=:\lim_{k\to\infty}I_k=0.$$
 \end{lem}
\begin{proof}
In order to prove the lemma we fix  $R>0$ (large). We split $B_{R_0}$ into
 $$A_1(R,y):=\{x\in B_{R_0}:|x_k+\rho_k\mu_kx-y|>R\rho_k\mu_k\},\quad A_2(R,y):=B_{R_0}\setminus A_1(R,y).$$  
 Then we can write $I_k=I_{1,k}+I_{2,k}$, where 
 $$I_{i,k}:=\int_{\R^{n}}K(y)e^{nw_k(y)}\int_{A_i(R,y)}\frac{\rho_k^2\mu_k^2}{|x_k+\rho_k\mu_kx-y|^2}dxdy,\quad i=1,2.$$ 
 Changing the variable $y\mapsto x_k+\rho_k\mu_ky$ and by Fubini's theorem one gets 
 \begin{align*}
  I_{2,k}&=\rho_k^n\int_{B_{R_0}}\int_{\R^n}K(x_k+\rho_k\mu_ky)e^{n\eta_k(y)}\frac{1}{|x-y|^2}\chi_{|x-y|\leq R}dydx\\
  &\leq \rho_k^n\int_{B_{R_0}}\int_{B_{R+R_0}}K(x_k+\rho_k\mu_ky)e^{n\eta_k(y)}\frac{1}{|x-y|^2}dydx\\
  &\leq C(n,\ve)\left(\sup_{B_{R+R_0+1}}Ke^{n\eta_k}\right)(R+R_0)^{n}R_0^{n-2}\rho_k^n,
 \end{align*}
  where $\eta_k(y):=w_k(x_k+\rho_k\mu_ky)-w_k(x_k)$. 
  If $x_k\to 0$ then $\eta_k\leq 0$ on $B_{R+R_0+1}$ for $k$ large. 
  Otherwise, for $k$ large $\rho_k\mu_ky\in B_1$ for every $y\in B_{R+R_0+1}$ and hence,  by Lemma \ref{local-bound}
  $$\eta_k(y)=v_k(x_k+\rho_k\mu_ky)-v_k(x_k)\leq C |\rho_k\mu_ky|^2t_k|\D v_k(0)|\leq C(R,R_0).$$ 
  Therefore, 
  $$\lim_{k\to\infty}I_{2,k}=0.$$
  Using the definition of $c_v$ we bound
 $$I_{1,k}\leq \frac{|B_{R_0}|}{R^2}\int_{\R^{n}}K(y)e^{nw_k(y)}dy\leq C(n,\kappa,R_0)\frac{1}{R^2}.$$
  Since $R>0$ is arbitrary, we conclude the lemma. 
\end{proof}

We need the following two lemmas only for $n$ odd.
\begin{lem}\label{odd}
Let $n\geq 5$. Let $v$ be given by \eqref{def-v}. For any $r>0$  and $\xi\in\R^n$ we set $$w(x)=v(rx+\xi),\quad x\in\R^n.$$ 
Then there exists $C>0$ (independent of $v,t,r,\xi$) such that  for every multi-index $\alpha\in \N^n$ with $|\alpha|=n-1$ we have
$\|D^\alpha w\|_{L_\frac12(\R^n)}\leq Ct(1+r^4|\D v(0)|)$. 
Moreover, for  any $\ve>0$ there exists $R>0$ (independent of $r$, $\xi$ and $t$) such that  
$$\int_{B_R^c}\frac{|D^\alpha w(x)|}{1+|x|^{n+1}}dx<\ve t(1+r^4|\D v(0)|),\quad |\alpha|=n-1.$$
\end{lem}
\begin{proof}
Differentiating under integral sign we obtain 
$$|D^\alpha w(x)|\leq Ct\int_{\R^n}\frac{r^{n-1}}{|rx+\xi-y|^{n-1}}f(y)dy+Ctr^4|\D v(0)|,\quad f(y):=K(y)e^{n(v(y)+c_{v})}.$$ 
If $n>5$ then the above inequality is true without the term $Ctr^4|\D v(0)|$.
Using a change of variable $y\mapsto \xi+ry$, we get 
\begin{align*}
 &\int_{\Omega}\frac{|D^\alpha w(x)|}{1+|x|^{n+1}}dx\\
 &\leq Ctr^n\int_{\R^n}f( \xi+ry)\int_{\Omega}\frac{1}{|x-y|^{n-1}}\frac{1}{1+|x|^{n+1}}dxdy+Ctr^4|\D v(0)|\int_{\Omega}\frac{dx}{1+|x|^{n+1}}.
 \end{align*}
 The lemma follows by taking $\Omega=\R^n$ or $B_R^c$.
\end{proof}

  \begin{lem}\label{L1/2conv}
 Let $\eta_k\to \eta$ in $C^{n-1}_{loc}(\R^n)$. We assume that for every $\ve>0$ there exists $R>0$  such that
 \begin{align}\label{L1/2}
  \int_{B_R^c}\frac{|\D^\frac{n-1}{2}\eta_k(x)|}{1+|x|^{n+1}}dx<\ve,\quad\text{for } k=1,2,\dots.
 \end{align}
    We further assume that 
 $$(-\D)^\frac n2\eta_k=K(x_k+\mu_kx)e^{n\eta_k}\quad\text{in }\R^n,\quad \int_{\R^n}|K(x_k+\mu_kx)|e^{n\eta_k(x)}dx\leq C,$$ where 
 $x_k\to x_\infty$, $\mu_k\to 0$, $K$ is a continuous function and $K(x_\infty)>0$.  Then $e^{n\eta}\in L^1(\R^n)$ and $\eta$ satisfies 
 $$(-\D)^\frac n2\eta=K(x_\infty)e^{n\eta}\quad\text{in }\R^n.$$
\end{lem}
\begin{proof}
First notice that $\D^\frac{n-1}{2}\eta_k\to \D^\frac{n-1}{2}\eta$ in $L_\frac 12(\R^n)$, thanks to \eqref{L1/2} and the convergence 
$\eta_k\to \eta$ in $C^{n-1}_{loc}(\R^n)$.

We claim that $\eta$ satisfies  $(-\D)^\frac n2\eta=K(x_\infty)e^{n\eta}$ in $\R^n$ in the sense of distribution. 

In order to prove the claim we let $\varphi\in C_c^\infty(\R^n)$. Then  
$$\lim_{k\to\infty}\int_{\R^n}K(x_k+\mu_kx)e^{n\eta_k(x)}\varphi(x)dx=\int_{\R^n}K(x_\infty)e^{n\eta(x)}\varphi(x)dx,$$ and 
$$\lim_{k\to\infty}\int_{\R^n}(-\D)^\frac{n-1}{2}\eta_k(-\D)^\frac 12\varphi dx=\int_{\R^n}(-\D)^\frac{n-1}{2}\eta(-\D)^\frac 12\varphi dx.$$
We conclude the claim.

To complete the lemma first notice that $e^{n\eta}\in L^1(\R^n)$, which follows from the fact that for any $R>0$ 
$$\int_{B_R}e^{n\eta}dx=\lim_{k\to\infty}\int_{B_R}e^{n\eta_k}dx=\lim_{k\to\infty}\int_{B_R}\frac{K(x_k+\mu_kx)}{K(x_\infty)}e^{n\eta_k(x)}dx\leq \frac{C}{K(x_\infty)}.$$
We fix a function $\psi\in C_c^\infty(B_2)$ such that $\psi=1$ on $B_1$. 
For $\varphi\in \mathcal{S}(\R^n)$ we set $\vp_k(x)=\vp(x)\psi(\frac{x}{k})$. The lemma follows by taking $k\to \infty$, thanks to the previous claim.
\end{proof}

\begin{lem}\label{T-compact}
 The operator $T: X\to X$ is compact. 
 \end{lem}
\begin{proof}
 Let $v_k$ be a bounded sequence in $X$. Then (up to a subsequence) $\{ v_k(0)\}$, $\{\D v_k(0)\}$, $\{A_{v_k}\}$ and $\{c_{v_k}\}$ are convergent sequences.
 Therefore, $|\D v_k(0)|(|x|^2-|x|^4)$ converges to some function in $X$. To conclude the lemma, it is sufficient to show that up to a subsequence
 $\{f_k\}$ converges in $X$, where $f_k$ is defined by 
 $$f_k(x)=\int_{\R^{n}}\log\left(\frac{1}{|x-y|}\right)Q(y)e^{nP(y)}e^{nP_{v_k}(y)}e^{n(v_k(y)+c_{v_k})}dy.$$
  Differentiating under integral sign one gets
  \begin{align*}
  |D^\beta f_k(x)|&\leq C\int_{\R^{n}}\frac{1}{|x-y|^{|\beta|}}Q(y)e^{nP(y)}e^{nP_{v_k}(y)}e^{n(v_k(y)+c_{v_k})}dy,\quad 0<|\beta|\leq n-1\\
  &\leq C\int_{\R^{n}}\frac{1}{|x-y|^{|\beta|}}e^{-|y|^4}dy\\
  &\leq C,
 \end{align*}
 where the second inequality follows from the  uniform bounds
 \begin{align}\label{bounds}
  |v_k(0)|\leq C,\,|c_{v_k}|\leq C,\, Qe^{nP}\leq C,\,\text{and }v_k(x)+P_{v_k}(x)\leq v_k(0)-|x|^4.
 \end{align}
 Indeed, for $0<|\beta|\leq n-1$
 $$\lim_{R\to\infty}\sup_{k}\sup_{x\in B_R^c}|D^\beta f_k(x)|=0,$$ and for every $0<s<1$ we have  $\|D^{n-1}f_k\|_{C^{0,s}(B_R)}\leq C(R,s)$. 
 Finally, using \eqref{bounds} we bound
$$
 |f_k(x)| \leq C\int_{\R^{n}}|\log|x-y||e^{-|y|^4}dy
 \leq C\log(2+|x|).
 $$
Thus, up to a subsequence, $f_k\to f$ in $C^{n-1}_{loc}(\R^n)$ for some $f\in C^{n-1}(R^n)$, and the global uniform estimates of $f_k$ 
and $D^\beta f_k$ would imply that $f_k\to f$ in $X$.
\end{proof}


\section{Proof of Theorem \ref{thm-2}}

 We consider the space
 $$X:=\left\{v\in C^{n-1}(\R^{n}): v\text{ is radially symmetric}, \|v\|_X<\infty\right\},$$ where 
 $$\|v\|_X:=\sup_{x\in\R^n}\left(\sum_{|\alpha|\leq 1}(1+|x|)^{|\alpha|-2}|D^\alpha v(x)|+\sum_{1<|\alpha|\leq n-1}|D^\alpha v(x)|\right).$$ 
  For $v\in X$, let $c_v$ be the constant determined by
  $$\int_{\R^{n}}Qe^{n(v+c_v)}dy=\kappa,$$ where $Q$ satisfies the hypothesis in Theorem \ref{thm-2}.  Without loss of generality we can assume that 
 $Q>0$ on $B_3$.
 
 We define an operator $T:X\to X$ given by $T(v)=\bar{v},$ where
 $$\bar{v}(x)=\frac{1}{\gamma_n}\int_{\R^{n}}\log\left(\frac{1}{|x-y|}\right)Q(y)e^{n(v(y)+c_v)}dy+
            \frac{1}{2n}|\D v(0)||x|^2.$$
   As in Lemma \ref{T-compact} one can show that the operator $T$ is compact.

 \medskip
 
 Proof of the following two  lemmas is similar to Lemmas \ref{lim-sup} and \ref{int1} respectively.
 \begin{lem}\label{4-1}
 Let $v$ solve $tT(v)=v$ for some $0<t\leq1$. Then  $\D v(0)<0$, and
\begin{align}\label{def-v4}
v(x)=\frac{t}{\gamma_n}\int_{\R^{n}}\log\left(\frac{1}{|x-y|}\right)Q(y)e^{n(v(y)+c_v)}dy +\frac{t}{2n}|\D v(0)||x|^2.
\end{align}
 \end{lem}

 \begin{lem}\label{int4}
  Let $(v_k,t_k)\in X\times (0,1]$ satisfy \eqref{def-v4} for all $k\in\N$. 
  Let  $x_k\in B_1$ be a point of maxima of $v_k$ on $\bar{B}_{|x_k|}$  and $v_k'(x_k)=0$.
  We set $w_k=v_k+c_{v_k}+\frac{1}{n}\log t_k$ and  $\mu_k=e^{-w_k(x_k)}$. 
  Let $\rho_k>0$ be such that   $\rho_k^2t_k\mu_k^2|\D v_k(0)|\leq C$ and $\rho_k\mu_k\to 0$. Then for any $R_0>0$
 $$\lim_{k\to\infty}\int_{\R^{n}}Q(y)e^{nw_k(y)}\int_{B_{R_0}}\frac{\rho_k^2\mu_k^2}{|x_k+\rho_k\mu_kx-y|^2}dxdy=0.$$
 \end{lem}
  
  Now we prove a similar local uniform upper bounds as in Lemma \ref{2nd-max}.
 \begin{lem}\label{2nd-max4}
Let $(v,t)\in X\times (0,1]$ satisfy \eqref{def-v4}.  Then there exists $C>0$ (independent of $v$ and $t$)
such that $$\sup_{B_\frac18} w\leq C,\quad w:=v+c_{v}+\frac{1}{n}\log t.$$
\end{lem}
\begin{proof}
 The proof is very similar to Lemma \ref{2nd-max}. Here we briefly sketch the proof. 
 
 We assume by contradiction that the conclusion of the lemma is false. Then there exists a sequence of $(v_k,t_k)$ and a sequence of points
 $x_k$ in $B_\frac14$ such that 
 $$w_k(x_k)\to\infty,\quad w_k\leq w_k(x_k)\text{ on } B_{|x_k|}, \quad x_k\text{ is a point of local maxima of }v_k.$$
 We set $\mu_k:=e^{-w_k(x_k)}$ and we distinguish following cases. 
 
  \medskip
\noindent
 \noindent\textbf{Case 1} Up to a subsequence $t_k\mu_k^2|\D v_k(0)|\to c_0\in [0,\infty).$ 

We set $\eta_k(x):=v_k(x_k+\mu_kx)-v_k(x_k).$ Then we have 
$$\int_{B_R}|\D\eta_k -t_k\mu_k^2|\D v_k(0)||dx\leq Ct_kR^{n-2}.$$
Now one can proceed exactly as in Case 1 in Lemma \ref{2nd-max}. 

 \medskip
\noindent
\noindent\textbf{Case 2} Up to a subsequence $t_k\mu_k^2|\D v_k(0)|\to \infty.$ 

We set $\psi_k(x)=v_k(x_k+\rho_k\mu_kx)-v_k(x_k)$ where $\rho_k$ is determined by $t_k\rho_k^2\mu_k^2|\D v_k(0)|=1$.
Then by Lemma \ref{int4} $$\int_{B_R}|\D\psi_k-1|dx=o(1),\quad\text{as }k\to\infty.$$ 
Similar to Case 2 in Lemma \ref{2nd-max} one can get a contradiction.
\end{proof}

With the help of Lemma \ref{2nd-max4} we prove

  \begin{lem}\label{cpt-4}
   There exists a constant $M>0$ such that for all $(v,t)\in X\times(0,1]$ satisfying \eqref{def-v4} we have $\|v\|\leq M$.
  \end{lem}
  \begin{proof}
  Let $(v,t)\in X\times(0,1]$ satisfies \eqref{def-v4}. We set $w:=v+c_v+\frac1n\log t$.  
  
  First we show that $|\D v(0)|\leq C$ for some $C>0$ independent of $v$ and $t$. Indeed,   
   differentiating under integral sign, from \eqref{def-v4}, and together with Lemma \ref{2nd-max4},  we get 
   \begin{align*}
    |\D v(0)|(1+t)&\leq C\int_{\R^n}\frac{1}{|y|^2}Q(y)e^{nw(y)}dy\\
    &=C\int_{B_\frac18}\frac{1}{|y|^2}Q(y)e^{nw(y)}dy+C\int_{B_\frac18^c}\frac{1}{|y|^2}Q(y)e^{nw(y)}dy\\
    &\leq C\int_{B_\frac18}\frac{1}{|y|^2}Q(y)dy+C\kappa\\
    &\leq C.
   \end{align*}
Hence $|\D v(0)|\leq C$. 

We define a function $\xi(x):=v(x)-\frac{t}{2n}|\D v(0)||x|^2$. Then $\xi$  is monotone decreasing on $(0,\infty)$, which follows from 
  the fact that $\D\xi\leq 0$. Therefore, 
  \begin{align*}
   w(x)&=\xi(x)+c_v+\frac1n\log t+\frac{t}{2n}|\D v(0)||x|^2\\
   &\leq \xi(\frac18)+c_v+\frac1n\log t+\frac{t}{2n}|\D v(0)||x|^2\\
   &\leq w(\frac18)+\frac{t}{2n}|\D v(0)||x|^2.
  \end{align*}
  Hence, $w(x)\leq \lambda(1+|x|^2)$  on $\R^n$ for some $\lambda>0$ independent of $v$ and $t$. Using this in \eqref{def-v4} one can show that 
     $$    |v(x)|\leq C \log(2+|x|)+C|x|^2,$$
     and   differentiating under integral sign, from \eqref{def-v4}
 $$
  |D^\beta v(x)|\leq C\int_{\R^n}\frac{1}{|x-y|^{|\beta|}}Q(y)e^{\lambda(1+|y|^2)}dy+C|D^\beta |x|^2|,\quad 0<|\beta|\leq n-1.\\
 $$
The lemma follows easily.
    \end{proof}

  \medskip\noindent\emph{Proof of Theorem \ref{thm-2}}          
By Schauder fixed point theorem (see Lemma \ref{leray}), the operator $T$ has a fixed point, thanks to Lemma \ref{cpt-4}. Let $v$ be a fixed point of $T$.
Then $u=v+c_v$ is a solution of \eqref{eq-1}. 

This finishes the proof of Theorem \ref{thm-2}.
\hfill $\square$

\medskip
    
    Now we prove the non existence results stated in the introduction.
   \begin{lem} \label{non}
     Let $n\in\{3,4\}$. Let $Q\in C^1_{rad}(\R^n)$ be monotone decreasing. We assume that 
     $$Q(x)= \delta e^{-\lambda|x|^2}\quad\text {for some } \delta>0\text{ and } \lambda>0,$$
     or $$Q(x)=e^{\xi(x)},\quad |x\cdotp\nabla Q(x)|\leq C,\quad \frac{\xi(x)}{|x|^2}\xrightarrow{|x|\to\infty}0.$$
    Then there is no radially symmetric solution to \eqref{eq-1} with  $\kappa>\Lambda_1$.
 \end{lem}
 \begin{proof}
  We assume by contradiction that there is a solution  $u$ to \eqref{eq-1} with $\kappa>\Lambda_1$, where $Q$ satisfies the 
  hypothesis of the lemma.
  
  We set $$v(x):=\frac{1}{\gamma_n}\int_{\R^n}\log\left(\frac{|y|}{|x-y|}\right)Q(y)e^{nu(y)}dy,\quad h:=u-v.$$
  Then $v(x)=-\frac{2\kappa}{\Lambda_1}\log|x|+o(\log|x|)$ as $|x|\to \infty$. Notice that
     $h$ is radially symmetric and $(-\D)^\frac n2h=0$ on $\R^n$. Therefore,  $h(x)=c_1+c_2|x|^2$ for some $c_1, c_2\in\R$. This follows easily if $n=4$. 
     For $n=3$, first notice that $\D h\in L_\frac12(\R^3)$. Hence, by  \cite[Lemma 15]{JMMX} 
  $\D h\equiv const$. Now radial symmetry of $h$ implies that  $h(x)=c_1+c_2|x|^2$.

      From a Pohozaev type identity in   \cite[Theorem 2.1]{Xu} we get
  \begin{align}\label{nabla}
   \frac{\kappa}{\gamma_n}\left(\frac{\kappa}{\gamma_n}-2\right)=\frac{1}{\gamma_n}\int_{\R^n}\left(x\cdotp\nabla K(x)\right)e^{nv(x)}dx,\quad K:=Qe^{nh}.
  \end{align}
Since $\kappa>\Lambda_1=2\gamma_n$, from \eqref{nabla} we deduce that $x\cdotp\nabla K(x)>0$ for some $x\in\R^n$. 
This implies that for $Q=\delta e^{-\lambda|x|^2}$ we must have $nc_2-\lambda>0$, which contradicts to the fact that  $Qe^{nu}\in L^1(\R^n)$. 
For $Q=e^{\xi}$, using that $Qe^{nu}\in L^1(\R^n)$ and that $\xi(x)=o(|x|^2)$ at infinity, one has $c_2\leq 0$.
Therefore, $x\cdotp\nabla K(x)\leq 0$ in $\R^n$, a contradiction.
 \end{proof}

Proof of the following lemma is similar to Lemma \ref{non}.
  
 \begin{lem} \label{non2}
     Let $\kappa\geq \Lambda_1$. Let $P$ be a non-constant and non-increasing radially symmetric polynomial of degree at most $n-1$. 
     Then there is no solution $u$ to \eqref{eq-2} (with $n\geq 3$) such that $u$ has the asymptotic behavior given by 
     $$u(x)=-\frac{2\kappa}{\Lambda_1}\log|x|+P(x)+o(\log|x|),\quad\text{as }|x|\to\infty.$$
 \end{lem}


\appendix
\section{Appendix}

\begin{lem}[Theorem 11.3 in \cite{G-T}]\label{leray}
 Let $T$ be a compact mapping of a Banach space $X$ into itself, and suppose that there exists a constant $M$ such that
 $$\|x\|_{X}<M$$
 for all $x\in X$ and $t\in (0,1]$ satisfying $tTx=x$. Then $T$ has a fixed point.
\end{lem}

\begin{lem}[\cite{LM}]
 Let $\D^mh=0$ in $B_{4R}\subset\R^n$. For any $x\in B_R$ and $0<r<R-|x|$ we have 
 \begin{align}\label{piz}
  \frac{1}{|B_r|}\int_{B_r(x)}h(z)dz=\sum_{i=0}^{m-1}c_ir^{2i}\D^ih(x),
 \end{align} 
 where
  $$c_0=1,\quad c_i=c(i,n)>0,\quad\text{for }i\geq 1.$$ Moreover, for every $k\geq0$ there exists $C=C(k,R)>0$ such that 
 \begin{align}\label{har-est}
  \|h\|_{C^k(B_R)}\leq C\|h\|_{L^1(B_{4R})}.
 \end{align}
\end{lem}

\begin{lem}\label{sch}
 Let $R>0$ and $B_R\subset\R^n$. Let $u_k\in C^{n-1,\alpha}(\R^n)$ for some $\alpha\in (\frac12,1)$ be such that 
 $$u_k(0)=0, \quad \| u_k^+\|_{L^\infty(B_R)}\leq C,\quad  \|(-\D)^\frac n2 u_k\|_{L^\infty(B_R)}\leq C,\quad \int_{B_R}|\D u_k|dx\leq C.$$
  If $n$ is an odd integer, we also assume that  $\|\D^\frac{n-1}{2}u_k\|_{L_\frac12(\R^n)}\leq C$.  
  Then (up to a subsequence) $u_k\to u$ in $C^{n-1}(B_\frac R8)$. 
\end{lem}
\begin{proof}
 First we prove the lemma for $n$ even. 
 
 We write $u_k=w_k+h_k$ where $$\left\{\begin{array}{ll}
                                 (-\D)^\frac n2 w_k=(-\D)^\frac n2 u_k\quad \text{in }B_R\\
                                 \D^jw_k=0,\quad \text{on }\partial B_R,\quad j=0,1,\dots, \frac{n-2}{2}.
                                \end{array}\right.
$$
Then by standard elliptic estimates, $w_k$'s are uniformly bounded in $C^{n-1,\beta}(B_R)$. Therefore, 
$$|h_k(0)|\leq C, \quad \| h_k^+\|_{L^\infty(B_R)}\leq C,\quad \int_{B_R}|\D h_k|dx\leq C.$$ 
Since $h_k$'s are $\frac n2$-harmonic, $\D h_k$'s are $(\frac n2-1)$-harmonic in $B_R$, and by \eqref{har-est} we obtain 
$$  \|\D h_k\|_{C^n(B_\frac R4)}\leq C\|\D h_k\|_{L^1(B_{R})}\leq C.$$ Using the identity \eqref{piz} we bound 
\begin{align*}
 \frac{1}{|B_R|}\int_{B_R(0)}h_k^-(z)dz&=\frac{1}{|B_R|}\int_{B_R(0)}h_k^+(z)dz-\frac{1}{|B_R|}\int_{B_R(0)}h_k(z)dz\\
 &= \frac{1}{|B_R|}\int_{B_R(0)}h_k^+(z)dz -h_k(0)-\sum_{i=1}^{m-1}c_iR^{2i}\D^ih_k(0)\\
 &\leq C,
\end{align*}
and hence $$\int_{B_R}|h_k(z)|dz=\int_{B_R}h_k^+(z)dz+\int_{B_R}h^-_k(z)dz\leq C.$$ 
Again by \eqref{har-est} we obtain 
$$\| h_k\|_{C^n(B_\frac R4)}\leq C\| h_k\|_{L^1(B_{R})}\leq C.$$ 
Thus, $u_k$'s are uniformly bounded in $C^{n-1,\beta}(B_\frac R4)$ 
and (up to a subsequence)
$u_k\to u$ in $C^{n-1}(B_\frac R4)$ for some $u\in C^{n-1}(B_\frac R4)$.

It remains to prove the lemma for $n$ odd. 

If $n$ is odd then $\frac{n-1}{2}$ is an integer. 
We split $\D^\frac{n-1}{2}u_k=w_k+h_k$ where $$\left\{\begin{array}{ll}
                                                (-\D)^\frac12w_k= (-\D)^\frac12\D^\frac{n-1}{2}u_k\quad\text{in }B_R\\
                                                w_k=0\quad\text{in }B_R^c.
                                               \end{array}\right.
$$
Then by Lemmas \ref{f-h} and \ref{f-s} one has $\|\D^\frac{n-1}{2}u_k\|_{C^\frac{1}{2}(B_\frac R2)}\leq C.$ Now one can proceed as in the case of 
even integer.
\end{proof}

\begin{lem}[\cite{JMMX}]\label{f-h}
 Let $u\in L_\sigma(\R^n)$ for some $\sigma\in (0,1)$ and $(-\D)^\sigma u=0$ in $B_{2R}$. Then for every $k\in \mathbb{N}$
 $$\|\nabla^ku\|_{C^0(B_R)}\leq C(n,\sigma,k)\frac{1}{R^k}\left(R^{2\sigma}\int_{\R^n\setminus B_{2R}}\frac{|u(x)|}{|x|^{n+2\sigma}}dx+
 \frac{\|u\|_{L^1(B_{2R})}}{R^n}\right)$$
 where $\alpha\in (0,1)$ and $k$ is an nonnegative integer.  
\end{lem}

\begin{lem}[\cite{Ros}]\label{f-s}
 Let $\sigma\in(0,1)$. Let $u$ be a solution of $$\left\{\begin{array}{ll}
                         (-\D)^\sigma u=f\quad\text{in }B_R\\
                                u=0\quad \text{in }B_R^c
                                                         \end{array}\right.
$$
Then $$\|u\|_{C^\sigma(\R^n)}\leq C(R,\sigma)\|f\|_{L^\infty(B_R)}.$$
\end{lem}

\noindent\textbf{Acknowledgements}
I would like to thank my advisor Prof. Luca Martinazzi for suggesting the problem and for many stimulating conversations.


\begin{thebibliography}{10}
\small

  
\bibitem{ARS} 
\textsc{Adimurthi, Robert, F., Struwe, M.:}
\emph{Concentration phenomena for Liouville's equation in dimension $4$},
J. Eur. Math. Soc. \textbf{8} (2006), 171-180.

\bibitem{CC} \textsc{S-Y. A. Chang; W. Chen:} \emph{A note on a class of higher order conformally covariant equations},
Discrete Contin. Dynam. Systems \textbf{63} (2001), 275-281.

\bibitem{Chanillo} \textsc{S. Chanillo, M. Kiessling:}
\emph{Surfaces with prescribed Gauss curvature},
Duke Math. J. \textbf{105} (2) (2000), 309-353.

\bibitem{Chen-Li}
 \textsc{W. Chen, C. Li:}
 \emph{Classification of solutions of some nonlinear elliptic equations}, Duke Math. J. \textbf{63} (1991) no. (3), 615-622. 
 
  \bibitem{DMR}
\textsc{F. Da Lio, L. Martinazzi, T. Rivi\`ere:}
\emph{Blow-up Analysis of a nonlocal Liouville-type equation}, 
Analysis and PDE \textbf{8} (2015) no. 7, 1757-1805.
 

\bibitem{Valdinoci} 
\textsc{E. Di Nezza, G. Palatucci, E. Valdinoci:}
\emph{Hitchhiker's guide to the fractional Sobolev spaces},
 Bull. Sci. Math. \textbf{136} (2012), no. 5, 521-573.

 \bibitem{dru} 
\textsc{O. Druet:}
\emph{Multibumps analysis in dimension $2$: quantification of blow-up levels}, Duke Math. J. \textbf{132} (2006), 217-269.
 
 \bibitem{G-T}
\textsc{D. Gilbarg; N. Trudinger:}
\emph{Elliptic partial differential equations of second order.}
Reprint of the 1998 edition. Classics in Mathematics. 
Springer-Verlag, Berlin, 2001. xiv+517 pp. ISBN: 3-540-41160-7.

 \bibitem{HD}
 \textsc{X. Huang, D. Ye:}
 \emph{Conformal metrics in $\mathbb{R}^{n}$ with constant $Q$-curvature and arbitrary volume},
 Calc. Var. Partial Differential Equations \textbf{54} (2015), 3373-3384.
 
 \bibitem{Hyd}
 \textsc{A. Hyder:} 
 \emph{Structure of conformal metrics on $\R^n$ with constant Q-curvature},  arXiv: 1504.07095 (2015).
 
 \bibitem{A-H}
 \textsc{A. Hyder:}
 \emph{Existence of entire solutions to a fractional Liouville equation in $\mathbb{R}^n$},
 Rend. Lincei Mat. Appl. \textbf{27} (2016), 1-14.
 
\bibitem{H-M}
\textsc{A. Hyder, L. Martinazzi:}
\emph{Conformal metrics on $\mathbb{R}^{n}$ with constant $Q$-curvature, prescribed volume and asymptotic behavior},
Discrete Contin. Dynam. Systems A \textbf{35} (2015), no.1, 283-299.

\bibitem{JMMX}
\textsc{T. Jin, A. Maalaoui, L. Martinazzi, J. Xiong:}
\emph{Existence and asymptotics for solutions of a non-local $Q$-curvature equation in dimension three},
 Calc. Var. Partial Differential Equations \textbf{52} (2015) no. 3-4, 469-488.

 \bibitem{Lin}
 \textsc{C. S. Lin:}
 \emph{A classification of solutions of a conformally invariant fourth order equation 
 in $\mathbb{R}^n$}, Comment. Math. Helv. \textbf{73} (1998), no. 2, 206-231.
 
  \bibitem{LM-Neg}
 \textsc{L. Martinazzi:}
 \emph{Conformal metrics on $\mathbb{R}^{2m}$ with constant $Q$-curvature},
 Rend. Lincei. Mat. Appl. \textbf{19} (2008), 279-292.
 
\bibitem{LM}
\textsc{L. Martinazzi:}
\emph{Classification of solutions to the higher order Liouville's equation on $\mathbb{R}^{n}$}, 
Math. Z. \textbf{263} (2009), no. 2, 307-329.

\bibitem{Mar1}
\textsc{L. Martinazzi:} 
\emph{Concentration-compactness phenomena in higher order Liouville's equation},
J. Functional Anal. \textbf{256} (2009), 3743-3771.


\bibitem{LM-vol}
\textsc{L. Martinazzi:}
\emph{Conformal metrics on $\mathbb{R}^{n}$ with constant $Q$-curvature and large volume},
Ann. Inst. Henri Poincar\'e (C) \textbf{30} (2013), 969-982.


  

\bibitem{Rob}
\textsc{F. Robert:} 
\emph{Concentration phenomena for a fourth order equation with exponential growth: the radial case}, 
J. Differential Equations \textbf{231}  (2006), no. $1$, 135-164.

\bibitem{Ros}
\textsc{X. Ros-Oton, J. Serra:}
\emph{The Dirichlet problem for the fractional Laplacian: Regularity up to the boundary},
 J. Math. Pures Appl. \textbf{101} (2014), no. 3, 275-302. 
 


\bibitem{W-Y}
\textsc{J. Wei, D. Ye:}
\emph{Nonradial solutions for a conformally invariant fourth order equation in  $\mathbb{R}^{4}$},
Calc. Var. Partial Differential Equations  \textbf{32} (2008), no. 3, 373-386.

\bibitem{Xu}
\textsc{X. Xu:} 
\emph{Uniqueness and non-existence theorems for conformally invariant equations}, 
J. Funct. Anal. \textbf{222} (2005), 1-28.

\end{thebibliography}
\end{document}